\documentclass[11pt]{amsart}

\usepackage[T1]{fontenc}
\usepackage{lmodern}
\usepackage{microtype}
\usepackage[a4paper,margin=29mm]{geometry}
\usepackage{amsmath,amssymb,amsthm,mathtools}
\usepackage{xcolor}
\usepackage{hyperref}

\numberwithin{equation}{section}

\hypersetup{
  colorlinks=true,
  linkcolor=blue!50!black,
  citecolor=blue!50!black,
  urlcolor=blue!50!black,
  pdftitle={A Newton Identity and Finite-Rank Reconstruction
  for the Queer Lie Superalgebra},
  pdfauthor={Abhishek Das and Santosha Pattanayak},
  pdfsubject={The center of the enveloping algebra of the queer Lie
  superalgebra}
}

\allowdisplaybreaks

\newtheorem{theorem}{Theorem}[section]
\newtheorem{proposition}[theorem]{Proposition}
\newtheorem{lemma}[theorem]{Lemma}
\newtheorem{corollary}[theorem]{Corollary}

\theoremstyle{definition}

\theoremstyle{remark}
\newtheorem{remark}[theorem]{Remark}
\newtheorem{example}[theorem]{Example}

\newcommand{\C}{\mathbb C}
\newcommand{\Q}{\mathbb Q}
\newcommand{\gl}{\mathfrak{gl}}
\newcommand{\q}{\mathfrak q}
\newcommand{\U}{\mathrm U}
\newcommand{\HC}{\operatorname{HC}}

\newcommand{\fall}[2]{(#1\mathbin{\downarrow}#2)}
\newcommand{\calD}{\mathcal D}
\newcommand{\bars}[1]{\overline{#1}}

\title[Newton identity for the queer Lie superalgebra]
{A Newton Identity and Finite-Rank Reconstruction
for the Queer Lie Superalgebra}

\author[A. Das]{Abhishek Das}
\address{
Department of Mathematics,
Indian Institute of Technology Kanpur,
Kanpur, Uttar Pradesh 208016, India
}
\email{abhidas20@iitk.ac.in}

\author[S. Pattanayak]{Santosha Pattanayak}
\address{
Department of Mathematics,
Indian Institute of Technology Kanpur,
Kanpur, Uttar Pradesh 208016, India
}
\email{santosha@iitk.ac.in}

\subjclass[2020]{Primary 17B35; Secondary 17B10, 05E05}

\keywords{Queer Lie superalgebra, center of an enveloping algebra,
Capelli element, factorial Schur \(Q\)-function, Newton identity,
Hankel determinant, resultant}

\begin{document}

\begin{abstract}
We establish a Newton-type identity for the queer Lie superalgebra
\(\mathfrak q_N\), relating Sergeev's odd cyclic central elements to
Nazarov's one-row Capelli elements.  The identity is obtained by
comparing Ivanov's generating function for factorial Schur
\(Q\)-functions with the queer Perelomov-Popov product of Grigoryev
and Nazarov.  Its coefficient expansion yields a triangular change of
generators between the odd cyclic and odd one-row families.  In
particular, the odd one-row Capelli elements generate the center, while
the even one-row elements are redundant.

In fixed rank, we derive determinantal relations and a generic
reconstruction theorem.  The basic cyclic Hankel determinant is
identified with a resultant and factored into the
failure-of-strong-typicality and shifted-resonance factors.  After
localization at this determinant, the center is generated by the first
\(2N\) odd cyclic elements; consequently, generic central characters
are determined by their values on these elements.
\end{abstract}

\maketitle

\section{Introduction}

Classical Newton identities relate power sums to elementary or complete
symmetric functions.  In enveloping algebras, their analogues compare
trace-type central elements with determinant-type Capelli elements.
For the general-linear Lie superalgebra \(\gl(m|n)\), the central
Newton relation between the coefficients of the Capelli Berezinian and
the Gelfand invariants follows from Nazarov's Liouville formula for the
quantum Berezinian \cite{Nazarov1991}.  More recently, Erat, Kannan,
and Kanungo gave an explicit formulation and derivation of this
relation \cite[Theorem~6.2.1]{EKK2025}.  Related recurrences, Hankel
determinants, resultants, and reconstruction formulas for rational
supermatrix characteristic functions were developed by Khudaverdian
and Voronov \cite{KV2005}, building on work of Kantor and Trishin
\cite{KT1997,KT1999}.

The queer case has a different invariant-theoretic structure: its
center is governed by supersymmetric functions and strict partitions.
Sergeev constructed cyclic central elements in \(\U(\q_N)\)
\cite{Sergeev1983}, Nazarov and Sergeev proved that the odd-indexed
cyclic elements generate the center
\cite[Proposition~1.1]{NazarovSergeev2006}.  The authors extended this
generation result to loop Lie superalgebras; specializing to
\(A=\C\) recovers the ordinary queer case
\cite[Theorem~5.3]{DP2025}.  Luo and Wang gave an independent
Schur--Weyl-theoretic construction of the cyclic generators
\cite[Section~3.2 and Proposition~3.12]{LuoWang2025}.

Nazarov constructed a distinguished Capelli basis indexed by strict
partitions and proved that its one-row elements generate the center
\cite[Proposition~4.8]{Nazarov1997}.  The eigenvalues of these elements
are described by factorial Schur \(Q\)-functions; see Ivanov
\cite{Ivanov2005}, Alldridge--Sahi--Salmasian \cite{ASS2018}, and the
precise Harish-Chandra normalization recorded by Kashuba and Molev
\cite[Remark~5.4, Equation~(5.5)]{KM2025}.  On the cyclic side,
Grigoryev and Nazarov obtained a queer Perelomov--Popov formula
expressing the central characters of the odd cyclic elements by a
rational product \cite[Section~2.2]{GN2019}.

The results of this paper have two complementary parts.  First, we show
that the factorial Schur \(Q\)-product and the queer
Perelomov--Popov product, although arising from different
constructions, are two forms of the same central Newton relation.
This identity yields an index-triangular change between the odd cyclic
and odd one-row Capelli families.  In particular, the odd one-row
elements generate the center, while the even one-row elements are
redundant.  Second, in fixed rank, we use the rationality of the
Harish-Chandra images to obtain determinantal relations, factor the
basic Hankel divisor, and reconstruct the localized center from
finitely many odd cyclic elements.

Fix \(N\geq1\), and write
\[
 Z_N=Z(\U(\q_N)).
\]
We use \(\lambda_1,\ldots,\lambda_N\) for the standard
Harish-Chandra coordinates.  Let \(c_{2r-1}\in Z_N\) denote Sergeev's
odd cyclic elements, and let \(C_{(r)}\) denote Nazarov's Capelli
element corresponding to the one-row strict partition \((r)\).  We
normalize
\[
 D_0=1,
 \qquad
 D_r=\frac{C_{(r)}}{r!}\quad(r\geq1),
\]
and form the reciprocal falling-factorial series
\[
 \calD_N(u)=\sum_{r\geq0}\frac{D_r}{\fall{u}{r}},
 \qquad
 \fall{u}{0}=1,
 \qquad
 \fall{u}{r}=u(u-1)\cdots(u-r+1)\quad(r\geq1).
\]
Our first main result identifies a consecutive quotient of this series
with the generating series of the cyclic elements.

\begin{theorem}[Main Newton identity]
\label{thm:intro-main}
In \(Z_N[[u^{-1}]]\), one has
\begin{equation}
 \boxed{
 \frac{\calD_N(u)}{\calD_N(u-1)}
 =
 1-\sum_{r\geq1}\frac{c_{2r-1}}{[u(u+1)]^r}.}
 \label{eq:intro-main}
\end{equation}
\end{theorem}

The appearance of the quadratic spectral variable \(u(u+1)\) is
essential.  The quotient of Ivanov's one-row product at consecutive
arguments becomes the Grigoryev--Nazarov product precisely after the
substitution
\[
 z=[u(u+1)]^{-1}.
\]
Injectivity of the Harish-Chandra homomorphism, applied
coefficientwise, then lifts the resulting equality of Harish-Chandra
images to the completed center.  Thus the queer identity is not a
formal specialization of the \(\gl(m|n)\) formula: both its spectral
parameter and its proof reflect the factorial Schur \(Q\)-structure of
the queer center.  To the best of our knowledge,
Theorem~\ref{thm:intro-main} has not appeared previously.

The first main consequences of Theorem~\ref{thm:intro-main} concern
the comparison of generators.  Expanding the reciprocal falling
factorials in ordinary powers of \(u^{-1}\) and comparing coefficients
gives an explicit recursion governed by Stirling numbers of the second
kind; see Proposition~\ref{prop:coefficient-recursion}.  In particular,
for every \(p\geq1\),
\begin{align*}
 D_{2p}
 &\in\Q[D_1,D_3,\ldots,D_{2p-1}],\\
 c_{2p-1}
 &=(2p-1)D_{2p-1}
   \pmod{\Q[D_1,D_3,\ldots,D_{2p-3}]}.
\end{align*}
Consequently,
\[
 \C[c_1,c_3,\ldots,c_{2p-1}]
 =
 \C[D_1,D_3,\ldots,D_{2p-1}]
 \qquad(p\geq1).
\]
Since the odd cyclic elements generate \(Z_N\), it follows that
\[
 Z_N=\C[D_1,D_3,D_5,\ldots],
\]
and every even one-row Capelli element is a polynomial in the preceding
odd ones.  We also give a filtered proof of the displayed generation
result that is independent of the Newton identity.  Nazarov proved
generation by the full one-row family
\cite[Proposition~4.8]{Nazarov1997}; the filtered argument given here
isolates the smaller odd-indexed family directly.

We next turn to our second main contribution, namely the fixed-rank
theory.  Put
\[
 b_r=c_{2r-1},
 \qquad
 \calD_N(u)=1+\sum_{r\geq1}\alpha_ru^{-r}.
\]
The rationality of the Harish-Chandra generating functions associated
with the sequences \((b_r)\) and \((\alpha_r)\) yields the central
Hankel relations
\[
 \det(b_{r+i+j})_{0\leq i,j\leq N}=0,
 \qquad
 \det(\alpha_{r+i+j})_{0\leq i,j\leq N}=0
 \qquad(r\geq1).
\]
For the coefficients \(D_r\) in the Newton basis, the corresponding
finite-rank relation instead takes the form of a difference
determinant.  This distinction reflects the fact that the ordinary
generating series
\(\sum_{r\geq0}D_rz^r\) is not rational, even when \(N=1\), while
multiplication by the spectral variable acts on the Newton
coefficients through a finite-difference operator.

Consider the basic \(N\times N\) cyclic Hankel determinant
\[
 \Delta_N=\det(b_{i+j-1})_{1\leq i,j\leq N}.
\]
Its factorization and its role in finite-rank reconstruction are
summarized in our second main theorem.

\begin{theorem}[Hankel divisor and generic reconstruction]
\label{thm:intro-reconstruction}
One has
\[
 \HC(\Delta_N)
 =
 (-1)^{\binom N2}
 \prod_{i,j=1}^N(\lambda_i+\lambda_j)
 \prod_{i\ne j}(\lambda_i-\lambda_j-1).
\]
In particular, \(\Delta_N\neq0\), and
\[
 Z_N[\Delta_N^{-1}]
 =
 \C[c_1,c_3,\ldots,c_{4N-1},\Delta_N^{-1}].
\]
Consequently, if \(\chi,\chi':Z_N\to\C\) are algebra homomorphisms
such that
\[
 \chi(\Delta_N)\neq0
 \quad\text{and}\quad
 \chi(c_{2r-1})=\chi'(c_{2r-1})
 \qquad(1\leq r\leq2N),
\]
then \(\chi=\chi'\).
\end{theorem}

Up to the displayed sign, \(\HC(\Delta_N)\) is the resultant of the
numerator and denominator of the queer Perelomov--Popov product.  The
factor
\(\prod_{i,j}(\lambda_i+\lambda_j)\) cuts out the
failure-of-strong-typicality locus, including the hyperplanes
\(\lambda_i=0\), whereas
\(\prod_{i\ne j}(\lambda_i-\lambda_j-1)\) cuts out the shifted-resonance
hyperplanes \(\lambda_i-\lambda_j=1\); see
\cite[Section~2]{FriskMazorchuk2009} for the distinction between
typical and strongly typical weights.  Hence, on the complement of
these two loci, the numerator and denominator are coprime and the
Hankel system determining the order-\(N\) recurrence is nonsingular.
In this sense, \(\Delta_N\) is the natural resultant divisor governing
generic reconstruction.  Although the underlying rational-sequence
mechanism is classical, the resulting central determinant identities,
the explicit factorization of the queer Hankel divisor, and the
localized reconstruction statement appear to be new for \(\q_N\).

The paper is organized as follows.
Section~\ref{sec:preliminaries} fixes our conventions for the queer
center, briefly recalls the general-linear comparison, introduces
Sergeev's cyclic elements and Nazarov's one-row Capelli elements, and
records the generating identities of Ivanov and
Grigoryev--Nazarov.  Section~\ref{sec:newton} proves the central Newton
identity and its symmetric product forms.
Section~\ref{sec:coefficients} derives the Stirling-number recursion
and the index-triangular comparison of the two odd families; it also
gives an independent filtered proof that the odd one-row Capelli
elements generate the center.  Finally,
Section~\ref{sec:hankel-reconstruction} develops the fixed-rank theory,
including the central Hankel identities, the difference-determinant
identity in the Newton basis, the resultant factorization of
\(\Delta_N\), and the reconstruction of the localized center.

\section{Preliminaries on the queer center}
\label{sec:preliminaries}

This section fixes our conventions for the queer Lie superalgebra,
Sergeev's cyclic central elements, the Harish-Chandra isomorphism, and
Nazarov's one-row Capelli elements.  We also record the two known
generating identities that will be compared in
Section~\ref{sec:newton}: Ivanov's one-row factorial Schur
\(Q\)-identity and the Grigoryev-Nazarov Perelomov-Popov product.

\subsection{The queer Lie superalgebra and Sergeev's cyclic elements}
\label{sec:cyclic}

We now fix the realization of \(\q_N\), normalize Sergeev's cyclic
elements, and recall the Harish--Chandra description of the center.
These are known inputs for the identities proved later.  Throughout,
the center is taken in the superalgebra sense; all its elements are even
\cite[Section~1]{NazarovSergeev2006}, so it agrees here with the usual
center.

Let
\[
 I_N=\{-N,\ldots,-1,1,\ldots,N\},
 \qquad
 \bars i=\begin{cases}0,&i>0,\\1,&i<0.\end{cases}
\]
If \(E_{ij}\) are the matrix units of \(\gl(N|N)\), put
\[
 F_{ij}=E_{ij}+E_{-i,-j}.
\]
The queer Lie superalgebra \(\q_N\) is the fixed-point subalgebra of
the involution \(E_{ij}\mapsto E_{-i,-j}\).  It is spanned by the
\(F_{ij}\), subject to \(F_{-i,-j}=F_{ij}\), and
\begin{align}
 [F_{ij},F_{kl}]
 &=\delta_{kj}F_{il}
 -(-1)^{(\bars i+\bars j)(\bars k+\bars l)}\delta_{il}F_{kj}
 +\delta_{k,-j}F_{-i,l}\notag\\
 &\quad
 -(-1)^{(\bars i+\bars j)(\bars k+\bars l)}
  \delta_{-i,l}F_{k,-j}.
 \label{eq:q-bracket}
\end{align}

For the standard triangular decomposition, let
\[
 \mathfrak h_{\bar 0}
 =\bigoplus_{i=1}^N\C F_{ii},
 \qquad
 \mathfrak h_{\bar 1}
 =\bigoplus_{i=1}^N\C F_{i,-i},
 \qquad
 \mathfrak h=\mathfrak h_{\bar0}\oplus\mathfrak h_{\bar1}.
\]
The positive and negative nilpotent subalgebras are \[\mathfrak n_+
 =
 \operatorname{span}_{\C}
 \{F_{ij},F_{i,-j}:1\leq i<j\leq N\},
\qquad
\mathfrak n_-
 =
 \operatorname{span}_{\C}
 \{F_{ij},F_{i,-j}:1\leq j<i\leq N\}\]
so that
\[
 \mathfrak q_N
 =\mathfrak n_-\oplus\mathfrak h\oplus\mathfrak n_+.
\]
We use the corresponding standard Borel subalgebra
\(\mathfrak b=\mathfrak h\oplus\mathfrak n_+\).

Let \(\varepsilon_1,\ldots,\varepsilon_N\) be the basis of
\(\mathfrak h_{\bar0}^*\) determined by
\[
 \varepsilon_i(F_{jj})=\delta_{ij}.
\]
Accordingly, a weight
\(\lambda\in\mathfrak h_{\bar0}^*\) is written as
\[
 \lambda=\lambda_1\varepsilon_1+\cdots+\lambda_N\varepsilon_N,
 \qquad
 \lambda_i=\lambda(F_{ii}).
\]
The nonzero even and odd roots are both
\(\varepsilon_i-\varepsilon_j\), \(i\neq j\).  Hence the even and odd
contributions to the Weyl vector cancel, and no \(\rho\)-shift occurs
in the Harish-Chandra coordinates used below.

For \(r\geq1\), define
\begin{equation}
 C_{ij}^{(r)}=
 \sum_{k_1,\ldots,k_{r-1}\in I_N}
 (-1)^{\bars{k_1}+\cdots+\bars{k_{r-1}}}
 F_{ik_1}F_{k_1k_2}\cdots F_{k_{r-1}j}.
 \label{eq:q-auxiliary}
\end{equation}
Thus \(C_{ij}^{(1)}=F_{ij}\) and
\begin{equation}
 C_{ij}^{(r+1)}
 =\sum_{k\in I_N}(-1)^{\bars k}F_{ik}C_{kj}^{(r)}.
 \label{eq:q-auxiliary-recursion}
\end{equation}
Replacing every summation index \(k_a\) in
\eqref{eq:q-auxiliary} by \(-k_a\), and using
\(F_{-a,-b}=F_{ab}\) and \(\bars{-a}=1-\bars a\), gives
\begin{equation}
 C_{-i,-j}^{(r)}=(-1)^{r-1}C_{ij}^{(r)}.
 \label{eq:q-auxiliary-reflection}
\end{equation}
Indeed, the \(r-1\) intermediate indices change the sign exponent by
\(r-1\) modulo two, and the ordered product of the \(F\)'s is otherwise
unchanged.

The cyclic elements are
\begin{equation}
 c_r=\sum_{i\in I_N}C_{ii}^{(r)}
 =\sum_{i_1,\ldots,i_r\in I_N}
 (-1)^{\bars{i_2}+\cdots+\bars{i_r}}
 F_{i_1i_2}\cdots F_{i_ri_1}.
 \label{eq:q-cyclic}
\end{equation}

The construction is due to Sergeev \cite{Sergeev1983}.  The centrality
and parity relation are reviewed in \cite[Section~1.2]{GN2019}, while
generation by the odd-indexed family is
\cite[Proposition~1.1]{NazarovSergeev2006} (see also \cite[Theorem~5.3]{DP2025}).

\begin{proposition}
\label{prop:cyclic-center}
For every \(r\geq1\), the element \(c_r\) is central and
\(c_{2r}=0\).  Moreover,
\begin{equation}
 Z(\U(\q_N))=\C[c_1,c_3,c_5,\ldots].
 \label{eq:q-center-cyclic}
\end{equation}
The equality means algebra generation; it does not assert algebraic
independence of the displayed infinite family.
\end{proposition}

\subsection{Supersymmetric functions and the Harish--Chandra map}
We next recall the supersymmetric-function realization of the queer
center.  With respect to the triangular decomposition fixed above, the
Harish-Chandra homomorphism identifies a central element with a
polynomial in the even Cartan coordinates
\(\lambda_1,\ldots,\lambda_N\).  Its image is characterized by
symmetry and the cancellation condition recalled below; see \cite[Definitions~2.1--2.2]{Ivanov2005} and
\cite[Chapter~III, Section~8]{Macdonald1995}.

For each \(N\geq1\), let
\begin{equation}
 \Gamma_N=\left\{
 f\in\C[\lambda_1,\ldots,\lambda_N]^{\mathfrak S_N}:
 f(t,-t,\lambda_3,\ldots,\lambda_N)
 \text{ is independent of }t
 \right\}.
 \label{eq:GammaN}
\end{equation}
By symmetry, the same cancellation condition holds in every pair of
variables.

Let \(\Lambda\) be the graded algebra of symmetric functions over
\(\C\) in the variables \(x_1,x_2,\ldots\), and let
\[
\rho_N:\Lambda\longrightarrow
\mathbb{C}[x_1,\ldots,x_N]^{\mathfrak S_N}
\]
denote the specialization obtained by setting \(x_i=0\) for \(i>N\).
The stable algebra of supersymmetric functions is
\[
\Gamma=
\left\{
f\in\Lambda:
\rho_N(f)(t,-t,x_3,\ldots,x_N)
\text{ is independent of \(t\) for every \(N\geq2\)}
\right\}.
\]
Equivalently,
\[
\Gamma\simeq
\bigoplus_{d\geq0}\varprojlim_N(\Gamma_N)_d,
\]
where the transition map \(\Gamma_N\to\Gamma_{N-1}\) is obtained by
setting \(x_N=0\).  Thus an element of \(\Gamma\) is a compatible stable
family of finite-variable supersymmetric polynomials.

For the standard upper-triangular Borel subalgebra, the
Harish--Chandra homomorphism is an algebra isomorphism
\begin{equation}
 \HC:Z(\U(\q_N))\xrightarrow{\ \sim\ }\Gamma_N.
 \label{eq:q-HC}
\end{equation}
See \cite[Section~2.3]{ChengWang2012},
\cite{Sergeev1999}, and \cite[Equation~(5.4)]{KM2025}.  We extend
\(\HC\) coefficientwise to central formal Laurent series.  This
extension is injective, because a series has zero image only if the
Harish--Chandra image of every coefficient is zero.

\subsection{One-row Capelli elements and their generating series}
\label{sec:capelli}
We next recall Nazarov's one-row Capelli elements and their
Harish--Chandra images.  Ivanov's generating identity packages these
elements into the central series entering the Newton relation.  The
factorial Schur \(Q\)-function input comes from Ivanov
\cite{Ivanov2005}, while the precise normalization used below is
recorded by Kashuba and Molev \cite{KM2025}.

For \(r\geq0\), write
\[
 \fall{u}{r}=u(u-1)\cdots(u-r+1),
 \qquad \fall{u}{0}=1.
\]
Recall that a strict partition is a finite sequence
\[
 \lambda=(\lambda_1>\lambda_2>\cdots>\lambda_{\ell}>0).
\]
Its size and length are denoted by
\[
 |\lambda|=\lambda_1+\cdots+\lambda_{\ell},
 \qquad
 \ell(\lambda)=\ell.
\]
Let \(Q^-_\lambda\) denote the factorial Schur \(Q\)-polynomial which
Ivanov denotes by \(Q^*_\lambda\).  Equivalently, it is the
multiparameter Schur \(Q\)-polynomial for \(a_j=j-1\).  For every
strict partition \(\lambda\), its highest homogeneous component is the
ordinary Schur \(Q\)-polynomial \(Q_\lambda\) by
\cite[Proposition~2.11]{Ivanov2005}.  Ivanov's stable basis theorem
\cite[Proposition~2.12]{Ivanov2005}, followed by specialization to
\(N\) variables, shows that the polynomials \(Q^-_\lambda\) with
\(\ell(\lambda)\leq N\) form a vector-space basis of \(\Gamma_N\); see
also \cite[Section~3]{KM2025}.

Ivanov's one-row generating function is
\begin{equation}
 \sum_{r\geq0}\frac{Q^-_{(r)}(\lambda)}{\fall{u}{r}}
 =\prod_{i=1}^N\frac{u+1+\lambda_i}{u+1-\lambda_i}
 \quad\text{in }\Gamma_N[[u^{-1}]],
 \qquad Q^-_{(0)}=1.
 \label{eq:Ivanov-one-row}
\end{equation}
This is \cite[Corollary~8.3]{Ivanov2005}, specialized to \(N\)
variables.

Let \(C_\lambda\in Z(\U(\q_N))\) be Nazarov's queer Capelli element
indexed by a strict partition \(\lambda\), and let \(g_\lambda\) be
the number of standard shifted tableaux of shape \(\lambda\).
Kashuba--Molev record the normalization
\begin{equation}
 \frac{g_\lambda}{|\lambda|!}\HC(C_\lambda)=Q^-_\lambda;
 \label{eq:capelli-HC-general}
\end{equation}
see \cite[Remark~5.4, Equation~(5.5)]{KM2025} and compare
\cite[Corollary~1.5]{ASS2018}.  Since \(g_{(r)}=1\), define
\begin{equation}
 D_0=1,\qquad D_r=\frac{1}{r!}C_{(r)}\quad(r\geq1).
 \label{eq:D-definition}
\end{equation}
Then
\begin{equation}
 \HC(D_r)=Q^-_{(r)}.
 \label{eq:D-HC}
\end{equation}

Nazarov observed that the full one-row family \(C_{(r)}\) generates the
center \cite[after Proposition~4.8]{Nazarov1997}. In Subsection~\ref{subsec:odd-generation} (see Theorem~\ref{thm:direct-odd-generation}) we show directly
that the odd-indexed one-row elements suffice.  

For a commutative algebra \(A\), all formal Laurent series below are
expanded at \(u=\infty\).  Thus \(A[[u^{-1}]]\) is equipped with its
\(u^{-1}\)-adic topology.  A series with constant term \(1\) is
invertible in this algebra, and quotients of such series are understood
in this formal sense.

We now encode the full one-row family in its natural reciprocal
falling-factorial series:
\begin{equation}
 \calD_N(u)=\sum_{r\geq0}\frac{D_r}{\fall{u}{r}}
 \in Z(\U(\q_N))[[u^{-1}]].
 \label{eq:capelli-series}
\end{equation}
For any fixed power of \(u^{-1}\), only finitely many summands
contribute, so \eqref{eq:capelli-series} is well-defined.  Equations
\eqref{eq:Ivanov-one-row} and \eqref{eq:D-HC} give
\begin{equation}
 \HC(\calD_N(u))
 =\prod_{i=1}^N\frac{u+1+\lambda_i}{u+1-\lambda_i}.
 \label{eq:capelli-series-HC}
\end{equation}

The Harish--Chandra product has a simple symmetry under the affine
reflection \(u\mapsto -u-2\).  This symmetry will later convert the
shifted quotient in the Newton identity into a symmetric product.

\begin{proposition}
\label{prop:reflection}
In \(Z(\U(\q_N))[[u^{-1}]]\),
\begin{equation}
 \calD_N(u)\calD_N(-u-2)=1.
 \label{eq:reflection}
\end{equation}
\end{proposition}

\begin{proof}
Substitution of \(-u-2\) in \eqref{eq:capelli-series-HC} gives
\[
 \HC(\calD_N(-u-2))
 =\prod_{i=1}^N\frac{u+1-\lambda_i}{u+1+\lambda_i}
 =\HC(\calD_N(u))^{-1}.
\]
The constant term of \(\calD_N(u)\) is one, so the inverse exists.
Coefficientwise injectivity of \(\HC\) proves \eqref{eq:reflection}.
\end{proof}

\subsection{The Grigoryev--Nazarov product}
\label{sec:pp}

We conclude the preliminaries by recording the rational generating
formula for the Harish-Chandra images of Sergeev's cyclic elements.
With the normalization of \(c_r\) used in
\eqref{eq:q-cyclic}, this is precisely the queer
Perelomov--Popov formula obtained by Grigoryev and Nazarov from their
Harish--Chandra recurrence
\cite[Proposition~5 and Section~2.2]{GN2019}. We record it in the exact
form needed for the Newton identity.

Setting \(j=i\) in \eqref{eq:q-auxiliary-reflection} and summing over
\(i\in I_N\) gives \(\HC(c_{2m+1})=2\sum_{i=1}^N a_i^{(m)}\), where
\(a_i^{(m)}=\HC(C^{(2m+1)}_{ii})\); the recurrence for \(a_i^{(m)}\) is
\cite[Proposition~5]{GN2019}, and summing it yields the following proposition. 

\begin{proposition}
\label{prop:PP-product}
In \(\Gamma_N[[s^{-2}]]\), one has
\begin{equation}
 1-\sum_{m\geq0}\HC(c_{2m+1})s^{-2m-2}
 =
 \prod_{i=1}^N
 \frac{s^2-\lambda_i(\lambda_i+1)}
      {s^2-\lambda_i(\lambda_i-1)}.
 \label{eq:PP-product}
\end{equation}
Equivalently, for a formal variable \(z\),
\begin{equation}
 1-\sum_{r\geq1}\HC(c_{2r-1})z^r
 =
 \prod_{i=1}^N
 \frac{1-z\lambda_i(\lambda_i+1)}
      {1-z\lambda_i(\lambda_i-1)}.
 \label{eq:PP-z}
\end{equation}
\end{proposition}

\section{The Newton identity and its product forms}
\label{sec:newton}

We now compare the two Harish--Chandra products recalled in
Subsections~\ref{sec:capelli} and~\ref{sec:pp}.  The quotient of
Ivanov's one-row product at consecutive arguments agrees with the
Grigoryev--Nazarov product after the quadratic substitution
\(z=[u(u+1)]^{-1}\).  Injectivity of the Harish--Chandra homomorphism
then lifts this equality to the completed center.  The reflection
identity subsequently gives the symmetric product forms.  All series
in this section are expanded at infinity; since \(\calD_N(u)\) has
constant term one, it is invertible in the \(u^{-1}\)-adic completion.

\begin{theorem}
\label{thm:queer-newton}
In \(Z(\U(\q_N))[[u^{-1}]]\),
\begin{equation}
 \boxed{
 \frac{\calD_N(u)}{\calD_N(u-1)}
 =1-\sum_{r\geq1}\frac{c_{2r-1}}{[u(u+1)]^r}.}
 \label{eq:queer-newton}
\end{equation}
\end{theorem}

\begin{proof}
Applying \(\HC\) to the left side and using
\eqref{eq:capelli-series-HC}, we obtain
\begin{align}
 \HC\left(\frac{\calD_N(u)}{\calD_N(u-1)}\right)
 &=\prod_{i=1}^N
 \frac{u+1+\lambda_i}{u+1-\lambda_i}
 \frac{u-\lambda_i}{u+\lambda_i}\notag\\
 &=\prod_{i=1}^N
 \frac{u(u+1)-\lambda_i(\lambda_i+1)}
      {u(u+1)-\lambda_i(\lambda_i-1)}.
 \label{eq:HC-newton-left}
\end{align}
In the formal identity \eqref{eq:PP-z}, substitute
\[
 z=[u(u+1)]^{-1}
 =u^{-2}(1+u^{-1})^{-1}\in u^{-2}\C[[u^{-1}]].
\]
The substitution is legitimate in the \(u^{-1}\)-adic topology:
the coefficient of any fixed power of \(u^{-1}\) receives
contributions from only finitely many powers of \(z\).  It gives
\begin{align}
 &1-\sum_{r\geq1}
 \HC(c_{2r-1})[u(u+1)]^{-r}\notag\\
 &\qquad=\prod_{i=1}^N
 \frac{u(u+1)-\lambda_i(\lambda_i+1)}
      {u(u+1)-\lambda_i(\lambda_i-1)}.
 \label{eq:HC-newton-right}
\end{align}
The right sides of \eqref{eq:HC-newton-left} and
\eqref{eq:HC-newton-right} coincide.  Thus the two sides of
\eqref{eq:queer-newton} have the same Harish--Chandra image.
Coefficientwise injectivity of \(\HC\) proves the theorem.
\end{proof}

The reflection identity turns the quotient into a symmetric product.

\begin{corollary}
\label{cor:newton-product}
In \(Z(\U(\q_N))[[u^{-1}]]\),
\begin{equation}
 \boxed{
 \calD_N(u)\calD_N(-u-1)
 =1-\sum_{r\geq1}\frac{c_{2r-1}}{[u(u+1)]^r}.}
 \label{eq:newton-product}
\end{equation}
\end{corollary}

\begin{proof}
Apply Proposition~\ref{prop:reflection} with \(u-1\) in place of
\(u\).  It gives
\[
 \calD_N(u-1)\calD_N(-u-1)=1,
\]
and therefore
\(\calD_N(u-1)^{-1}=\calD_N(-u-1)\).  Substitute this equality in
the left side of \eqref{eq:queer-newton}.
\end{proof}

A shift by one half converts the quadratic parameter \(u(u+1)\) into
a square and gives a more symmetric form of the identity.

\begin{corollary}
\label{cor:symmetric-parameter}
Let
\[
 v=s\left(1+\frac{1}{4s^2}\right)^{1/2}
 =s+\frac{1}{8s}-\frac{1}{128s^3}+\cdots,
 \qquad v^2=s^2+\frac14.
\]
Then, in \(Z(\U(\q_N))[[s^{-1}]]\),
\begin{equation}
 \boxed{
 1-\sum_{m\geq0}\frac{c_{2m+1}}{s^{2m+2}}
 =\calD_N\left(v-\tfrac12\right)
  \calD_N\left(-v-\tfrac12\right).}
 \label{eq:symmetric-parameter}
\end{equation}
\end{corollary}

\begin{proof}
Set \(u=v-\tfrac12\) in \eqref{eq:newton-product}.  Since
\[
 u(u+1)=v^2-\frac14=s^2,
 \qquad -u-1=-v-\frac12,
\]
the result is exactly \eqref{eq:symmetric-parameter}.
\end{proof}

\section{Coefficient recursion and triangular consequences}
\label{sec:coefficients}

This section extracts ordinary Laurent coefficients from the Newton
identity.  Reciprocal falling factorials are converted to powers of
\(u^{-1}\) by Stirling numbers, producing an explicit recursion and the
triangular change of generators stated in the introduction.

\subsection{Stirling expansion and triangularity}
\label{subsec:stirling-triangularity} For \(n,r\geq0\), the Stirling number of the second kind \(S(n,r)\) is
the number of partitions of \(\{1,\ldots,n\}\) into \(r\) nonempty
blocks. We use the conventions
\[
S(0,0)=1,\qquad
S(n,0)=0\quad(n>0),\qquad
S(n,r)=0\quad\text{if }r<0\text{ or }r>n.
\]
They satisfy
\[
S(n,r)=S(n-1,r-1)+rS(n-1,r)
\qquad(n,r\geq1),
\]
and their ordinary generating function is
\[
\sum_{n\geq r}S(n,r)z^n
=
\frac{z^r}{(1-z)(1-2z)\cdots(1-rz)}.
\]

Set
\begin{equation}
 \beta_0=1,
 \qquad
 \beta_j=\sum_{k=1}^jS(j,k)D_k\quad(j\geq1).
 \label{eq:beta-definition}
\end{equation}

We now extract the coefficient of \(u^{-n}\) from the Newton identity.
The resulting recursion gives an explicit comparison between the
Capelli coefficients and the odd cyclic elements.

\begin{proposition}
\label{prop:coefficient-recursion}
For every \(n\geq1\),
\begin{equation}
 \boxed{
 \sum_{j=1}^{n-1}jS(n-1,j)D_j
 =\sum_{r=1}^{\lfloor n/2\rfloor}\ \sum_{m=2r}^{n}
 (-1)^m\binom{m-r-1}{r-1}
 \beta_{n-m}c_{2r-1}.}
 \label{eq:coefficient-recursion}
\end{equation}
For \(n=1\), both sides are zero.
\end{proposition}

\begin{proof}
We first establish the three Laurent expansions used in the
coefficient comparison.  For \(r\geq1\),
\begin{align}
 \frac{1}{\fall{u}{r}}
 &=u^{-r}\prod_{a=1}^{r-1}(1-au^{-1})^{-1}
 =\sum_{n\geq r}S(n-1,r-1)u^{-n},
 \label{eq:falling-u-expansion}\\
 \frac{1}{\fall{u-1}{r}}
 &=u^{-r}\prod_{a=1}^{r}(1-au^{-1})^{-1}
 =\sum_{n\geq r}S(n,r)u^{-n}.
 \label{eq:falling-u-minus-one-expansion}
\end{align}
The last equalities follow, for example, by induction on \(r\) from
the recurrence for \(S(n,r)\); equivalently, they are the ordinary
generating functions for the Stirling numbers.  Here
\(\fall{u-1}{r}=(u-1)(u-2)\cdots(u-r)\).

Write
\[
 \calD_N(u)=\sum_{n\geq0}\alpha_nu^{-n},
 \qquad
 \calD_N(u-1)=\sum_{n\geq0}\beta_nu^{-n}.
\]
Equations \eqref{eq:falling-u-expansion} and
\eqref{eq:falling-u-minus-one-expansion} give
\begin{equation}
 \alpha_0=\beta_0=1,
 \qquad
 \alpha_n=\sum_{j=1}^nS(n-1,j-1)D_j,
 \qquad
 \beta_n=\sum_{j=1}^nS(n,j)D_j.
 \label{eq:alpha-beta}
\end{equation}
The last expression agrees with \eqref{eq:beta-definition}.  Using
the Stirling recurrence in \eqref{eq:alpha-beta}, we obtain
\begin{equation}
 \beta_n-\alpha_n
 =\sum_{j=1}^{n-1}jS(n-1,j)D_j.
 \label{eq:beta-minus-alpha}
\end{equation}

The binomial series supplies the third expansion:
\begin{align}
 [u(u+1)]^{-r}
 &=u^{-2r}(1+u^{-1})^{-r}\notag\\
 &=\sum_{m\geq2r}(-1)^m
   \binom{m-r-1}{r-1}u^{-m}.
 \label{eq:quadratic-expansion}
\end{align}
Indeed, on writing \(m=2r+k\), the coefficient becomes
\((-1)^k\binom{r+k-1}{k}\), which equals the coefficient displayed
in \eqref{eq:quadratic-expansion} because \(2r\) is even and
\(\binom{r+k-1}{k}=\binom{m-r-1}{r-1}\).

Multiply \eqref{eq:queer-newton} by \(\calD_N(u-1)\):
\begin{equation}
 \calD_N(u)=\calD_N(u-1)
 \left(1-\sum_{r\geq1}c_{2r-1}[u(u+1)]^{-r}\right).
 \label{eq:newton-multiplied}
\end{equation}
Using \eqref{eq:quadratic-expansion}, the coefficient of \(u^{-n}\)
on the right side of \eqref{eq:newton-multiplied} is
\[
 \beta_n-
 \sum_{r=1}^{\lfloor n/2\rfloor}\ \sum_{m=2r}^{n}
 (-1)^m\binom{m-r-1}{r-1}\beta_{n-m}c_{2r-1}.
\]
Equate this with the coefficient \(\alpha_n\) on the left and apply
\eqref{eq:beta-minus-alpha}.  This proves
\eqref{eq:coefficient-recursion}.
\end{proof}

The recursion is triangular with respect to increasing index.  Its
even and odd specializations successively eliminate the even Capelli
elements and compare the leading terms of the two odd families.

\begin{theorem}
\label{thm:triangularity}
For every \(p\geq1\),
\begin{align}
 D_{2p}&\in\Q[D_1,D_3,\ldots,D_{2p-1}],
 \label{eq:even-D-redundant}\\
 c_{2p-1}&=(2p-1)D_{2p-1}
 +P_p(D_1,D_3,\ldots,D_{2p-3})
 \label{eq:c-D-triangular}
\end{align}
for a polynomial \(P_p\) with rational coefficients.  Consequently,
for every \(p\geq1\),
\begin{equation}
 \C[c_1,c_3,\ldots,c_{2p-1}]
 =\C[D_1,D_3,\ldots,D_{2p-1}].
 \label{eq:finite-subalgebras}
\end{equation}
Thus the two odd-indexed families are related by a triangular change
of algebra generators with respect to increasing index. 
\end{theorem}

\begin{proof}
We prove \eqref{eq:even-D-redundant} and
\eqref{eq:c-D-triangular} simultaneously by induction on \(p\).
For \(n=2\), recursion \eqref{eq:coefficient-recursion} reads
\[
 D_1=c_1.
\]
For \(n=3\), it reads
\[
 D_1+2D_2=c_1(D_1-1).
\]
Hence
\begin{equation}
 c_1=D_1,
 \qquad
 D_2=\frac12D_1^2-D_1,
 \label{eq:base-low}
\end{equation}
which proves both assertions for \(p=1\).

Assume they have been proved through \(p-1\).  Put \(n=2p\) in
\eqref{eq:coefficient-recursion}.  On the right side, the only term
involving \(c_{2p-1}\) has \(r=p\) and \(m=2p\), and its coefficient is
\[
 (-1)^{2p}\binom{p-1}{p-1}\beta_0=1.
\]
On the left side, the coefficient of \(D_{2p-1}\) is
\[
 (2p-1)S(2p-1,2p-1)=2p-1.
\]
Every remaining \(D_j\) has \(j\leq2p-2\), and every remaining cyclic
element is among \(c_1,c_3,\ldots,c_{2p-3}\).  The induction
hypotheses replace all lower even \(D_j\)'s and lower cyclic elements
by polynomials in \(D_1,D_3,\ldots,D_{2p-3}\).  Solving for
\(c_{2p-1}\) proves \eqref{eq:c-D-triangular} at index \(p\).

Next put \(n=2p+1\).  The coefficient of \(D_{2p}\) on the left side
of \eqref{eq:coefficient-recursion} is
\[
 2pS(2p,2p)=2p.
\]
Every other \(D_j\) there has \(j\leq2p-1\).  On the right side one
has \(r\leq p\), so only \(c_1,c_3,\ldots,c_{2p-1}\) occur.  Moreover,
\(m\geq2\), and therefore the factor \(\beta_{2p+1-m}\) involves only
\(D_1,\ldots,D_{2p-1}\).  The induction hypotheses, together with the
newly established formula for \(c_{2p-1}\), express every term other
than \(2pD_{2p}\) as a polynomial in
\(D_1,D_3,\ldots,D_{2p-1}\) with rational coefficients.  Dividing by
\(2p\) proves \eqref{eq:even-D-redundant} and completes the induction.

Equation \eqref{eq:c-D-triangular} immediately gives
\[
 \C[c_1,c_3,\ldots,c_{2p-1}]
 \subseteq\C[D_1,D_3,\ldots,D_{2p-1}].
\]
For the reverse inclusion, solve \eqref{eq:c-D-triangular} for the
newest odd Capelli element:
\[
 D_{2p-1}=\frac{1}{2p-1}
 \left(c_{2p-1}-P_p(D_1,D_3,\ldots,D_{2p-3})\right).
\]
Induction on \(p\), beginning with \(D_1=c_1\), expresses each lower
odd \(D_j\) in terms of lower odd cyclic elements.  This proves
\eqref{eq:finite-subalgebras}.
\end{proof}

\begin{example}
The cases \(n=2,3,4,5\) of
\eqref{eq:coefficient-recursion}, before final substitution, are
\begin{align*}
 c_1&=D_1,\\
 D_1+2D_2&=c_1(D_1-1),\\
 c_3&=6D_2+3D_3-D_1D_2,\\
 D_1+14D_2+18D_3+4D_4
 &=c_1(D_1+2D_2+D_3-1)+(D_1-2)c_3.
\end{align*}
Consequently,
\begin{align*}
 D_2&=\frac12D_1^2-D_1,\\
 c_3&=3D_3+4D_1^2-6D_1-\frac12D_1^3,\\
 D_4&=D_1D_3-6D_3+6D_1-\frac{11}{2}D_1^2
      +\frac32D_1^3-\frac18D_1^4.
\end{align*}
These examples display both the triangular leading coefficient and
the naturally occurring rational coefficients.
\end{example}

\subsection{An independent proof of odd one-row generation}\label{subsec:odd-generation}

The triangularity theorem, together with Sergeev's generation theorem,
already implies that the odd-indexed one-row Capelli elements generate
the center.  We conclude this section with a direct filtered proof that
does not use the Newton identity or the Grigoryev--Nazarov product.
This independent argument explains the result entirely through the
structure of the supersymmetric-function algebra.

\begin{theorem}
\label{thm:direct-odd-generation}
The odd-indexed normalized one-row Capelli elements generate the
center:
\[
 Z(\U(\q_N))=\C[D_1,D_3,D_5,\ldots].
\]
\end{theorem}

\begin{proof}
We first work in the stable algebra \(\Gamma\) of supersymmetric
functions.  Write \(p_k=\sum_i x_i^k\) and
\(q_r=Q_{(r)}\), with \(q_0=1\).  The standard one-row generating
function is
\begin{equation}
 \sum_{r\geq0}q_rt^r
 =\prod_i\frac{1+x_it}{1-x_it}
 =\exp\left(2\sum_{\substack{k\geq1\\k\text{ odd}}}
            \frac{p_k}{k}t^k\right).
 \label{eq:ordinary-Q-generating}
\end{equation}
The first equality is classical; the second follows by taking formal
logarithms and using
\[
 \log\frac{1+z}{1-z}
 =2\sum_{j\geq0}\frac{z^{2j+1}}{2j+1}.
\]
Taking the coefficient of \(t^{2m-1}\) in
\eqref{eq:ordinary-Q-generating} gives
\begin{equation}
 q_{2m-1}=\frac{2}{2m-1}p_{2m-1}
 +R_m(p_1,p_3,\ldots,p_{2m-3})
 \label{eq:q-p-triangular}
\end{equation}
for a polynomial \(R_m\) with rational coefficients.  Indeed, the
factor \(p_{2m-1}t^{2m-1}\) can occur only once and only from the
linear term of the exponential; every other contribution uses odd
indices strictly smaller than \(2m-1\).

It is standard that
\begin{equation}
 \Gamma=\C[p_1,p_3,p_5,\ldots];
 \label{eq:Gamma-power-sums}
\end{equation}
see \cite[Chapter~III, Section~8]{Macdonald1995}.  Equation
\eqref{eq:q-p-triangular} and induction on \(m\) therefore imply
\begin{equation}
 \Gamma=\C[q_1,q_3,q_5,\ldots].
 \label{eq:Gamma-odd-q}
\end{equation}

Specialization to the first \(N\) variables maps \(\Gamma\) onto
\(\Gamma_N\).  One way to see surjectivity is that the specializations
of the \(Q_\lambda\) with \(\ell(\lambda)\leq N\) form a basis of
\(\Gamma_N\).  Thus the specialized odd \(q_r\)'s generate
\(\Gamma_N\).

We now pass from ordinary to factorial one-row functions.  Filter
\(\Gamma_N\) by total polynomial degree and put
\[
 A_N=\C[Q^-_{(1)},Q^-_{(3)},Q^-_{(5)},\ldots]\subseteq\Gamma_N.
\]
The highest homogeneous component of \(Q^-_{(r)}\) is \(q_r\).
We prove \(A_N=\Gamma_N\) by induction on degree.  Let
\(f\in\Gamma_N\) have degree at most \(d\), and let \(f_d\) be its
degree-\(d\) component.  The cancellation condition is homogeneous,
so \(f_d\in\Gamma_N\).  By the preceding paragraph, there is a
weighted-homogeneous polynomial \(F\) such that
\[
 f_d=F(q_1,q_3,q_5,\ldots),
\]
where \(q_r\) has weight \(r\).  Hence
\(F(Q^-_{(1)},Q^-_{(3)},\ldots)\) has highest homogeneous component
\(f_d\).  The difference
\[
 f-F(Q^-_{(1)},Q^-_{(3)},\ldots)
\]
has degree at most \(d-1\), and belongs to \(A_N\) by the induction
hypothesis.  The second term belongs to \(A_N\) by definition, so
\(f\in A_N\).  The degree-zero case is immediate.  This proves
\(A_N=\Gamma_N\).

Finally, \eqref{eq:D-HC} gives
\[
 \HC\bigl(\C[D_1,D_3,D_5,\ldots]\bigr)=\Gamma_N.
\]
Surjectivity and injectivity of the Harish-Chandra isomorphism
\eqref{eq:q-HC} prove the asserted equality. 
\end{proof}

\begin{remark}
Theorem~\ref{thm:direct-odd-generation} also follows immediately from
Proposition~\ref{prop:cyclic-center} and
\eqref{eq:finite-subalgebras}.  The proof above is independent of the
Newton identity and of the Perelomov-Popov formula.
\end{remark}

\section{Finite-rank identities and generic reconstruction}
\label{sec:hankel-reconstruction}

This section develops the fixed-rank consequences of the two rational
Harish-Chandra products.  We prove central Hankel identities, a
difference determinant adapted to the Newton basis, an explicit
resultant formula for the basic Hankel divisor, and finite generation
after localization.  The general rational-sequence results underlying
these arguments are cited at the points where they are used.

Throughout the section, we put
\[
 Z_N=Z\bigl(\U(\q_N)\bigr),
 \qquad
 b_r=c_{2r-1}\quad(r\geq1),
 \qquad
 \pi_r=\HC(b_r),
\]
and let \(K=\C(\lambda_1,\ldots,\lambda_N)\).  By
Proposition~\ref{prop:cyclic-center} (or, equivalently, by
Theorems~\ref{thm:direct-odd-generation} and~\ref{thm:triangularity}),
\begin{equation}
 Z_N=\C[b_1,b_2,b_3,\ldots].
 \label{eq:center-generated-by-b}
\end{equation}
Since \(\HC(Z_N)=\Gamma_N\subseteq
\C[\lambda_1,\ldots,\lambda_N]\), the algebra \(Z_N\) is an integral
domain, and the injective homomorphism \(\HC\) extends uniquely to an
injection
\[
 \HC:\operatorname{Frac}(Z_N)\hookrightarrow K.
\]

\subsection{Rational series and finite-rank identities}

We begin with the elementary rational-series lemma underlying all the
Hankel identities below.  In the supermatrix setting, the same
recurrence and determinant argument appears in
\cite[Eq.~(3.1) and Corollary~(3.8)]{KV2005}; the general formulation in
terms of recurrent sequences and Hankel matrices is recorded in
\cite[Appendix~A]{KV2005}.

\begin{lemma}
\label{lem:rational-hankel}
Let \(A\) be a commutative ring, and suppose that
\[
 H(z)=\sum_{m\geq0}h_mz^m\in A[[z]]
\]
satisfies \(Q(z)H(z)=P(z)\), where
\[
 \deg P\leq N,
 \qquad
 Q(z)=1+q_1z+\cdots+q_Nz^N,
 \qquad
 \deg Q\leq N.
\]
Then
\begin{equation}
 h_m+q_1h_{m-1}+\cdots+q_Nh_{m-N}=0
 \qquad(m>N).
 \label{eq:rational-linear-recurrence}
\end{equation}
Consequently,
\begin{equation}
 \det\bigl(h_{r+i+j}\bigr)_{0\leq i,j\leq N}=0
 \qquad(r\geq1).
 \label{eq:abstract-hankel-vanishing}
\end{equation}
\end{lemma}

\begin{proof}
For \(m>N\), the coefficient of \(z^m\) in \(P(z)\) is zero.
Equating the coefficient of \(z^m\) in \(Q(z)H(z)=P(z)\) therefore
gives \eqref{eq:rational-linear-recurrence}.

Fix \(r\geq1\).  For each \(0\leq i\leq N\), apply
\eqref{eq:rational-linear-recurrence} with \(m=r+i+N>N\).  We obtain
\[
 h_{r+i+N}
 =-q_1h_{r+i+N-1}-\cdots-q_Nh_{r+i}.
\]
Thus the last column of the matrix
\(\bigl(h_{r+i+j}\bigr)_{0\leq i,j\leq N}\) is an
\(A\)-linear combination of its first \(N\) columns, with coefficients
independent of the row.  Its determinant is therefore zero.
\end{proof}

The odd cyclic series is rational by the queer Perelomov-Popov formula
\cite[Sec.~2.2]{GN2019}, while the Laurent expansion of the one-row
Capelli series is rational by Ivanov's generating identity
\cite[Cor.~8.3]{Ivanov2005} and the Harish--Chandra calculation recalled
in \eqref{eq:capelli-series-HC}.  Applying
Lemma~\ref{lem:rational-hankel} to these two series gives the following
central identities.

\begin{proposition}
\label{prop:finite-rank-hankel}
Write
\[
 \calD_N(u)=1+\sum_{n\geq1}\alpha_nu^{-n}.
\]
Equivalently, by \eqref{eq:alpha-beta},
\[
 \alpha_n=\sum_{j=1}^n S(n-1,j-1)D_j
 \qquad(n\geq1).
\]
Then, for every \(r\geq1\), the following identities hold in \(Z_N\):
\begin{align}
 \det\bigl(c_{2(r+i+j)-1}\bigr)_{0\leq i,j\leq N}&=0,
 \label{eq:cyclic-hankel}\\
 \det\bigl(\alpha_{r+i+j}\bigr)_{0\leq i,j\leq N}&=0.
 \label{eq:Capelli-hankel}
\end{align}
\end{proposition}

\begin{proof}
For the cyclic elements, substitute \(z=s^{-2}\) in the
Perelomov--Popov identity \eqref{eq:PP-product}.  This gives
\begin{equation}
 1-\sum_{r\geq1}\pi_rz^r
 =\prod_{i=1}^{N}
 \frac{1-\lambda_i(\lambda_i+1)z}
      {1-\lambda_i(\lambda_i-1)z}.
 \label{eq:hankel-PP-product}
\end{equation}
The numerator and denominator have degree at most \(N\), and the
denominator has constant term one.  Lemma~\ref{lem:rational-hankel}
therefore applies with \(h_0=1\) and \(h_r=-\pi_r\) for \(r\geq1\).
Since every index \(r+i+j\) in the relevant matrix is positive,
\[
 \det(h_{r+i+j})_{0\leq i,j\leq N}
 =(-1)^{N+1}
  \det(\pi_{r+i+j})_{0\leq i,j\leq N}=0.
\]
But
\[
 \HC\!\left(
  \det(b_{r+i+j})_{0\leq i,j\leq N}
 \right)
 =\det(\pi_{r+i+j})_{0\leq i,j\leq N}.
\]
The injectivity of \(\HC\) proves \eqref{eq:cyclic-hankel}.

For the Capelli coefficients, put \(z=u^{-1}\) in
\eqref{eq:capelli-series-HC}.  After multiplying the numerator and
denominator of each factor by \(z\), we obtain
\begin{equation}
 \HC\bigl(\calD_N(z^{-1})\bigr)
 =\prod_{i=1}^{N}
   \frac{1+(1+\lambda_i)z}{1+(1-\lambda_i)z}
 =1+\sum_{n\geq1}\HC(\alpha_n)z^n.
 \label{eq:Capelli-rational-z}
\end{equation}
This is again a quotient of two polynomials of degree at most \(N\),
with denominator having constant term one.  Hence
Lemma~\ref{lem:rational-hankel} gives
\[
 \det\bigl(\HC(\alpha_{r+i+j})\bigr)_{0\leq i,j\leq N}=0.
\]
Since \(\HC\) is an algebra homomorphism and is injective, this is
equivalent to \eqref{eq:Capelli-hankel}.
\end{proof}

\subsection{A difference-determinant identity in the Newton basis}

The Capelli series is naturally expanded in the reciprocal
falling-factorial, or Newton, basis used in Ivanov's one-row identity
\cite[Sec.~8]{Ivanov2005}.  The ordinary Laurent coefficients
\(\alpha_n\) satisfy the constant-coefficient recurrence of the previous
subsection.  The original one-row coefficients \(D_j\), however, are
governed by a recurrence involving the shift operator and polynomial
coefficients in \(j\).  We now formulate this recurrence carefully and
eliminate its noncentral coefficients by a determinant.

For \(j\geq0\), set
\[
 \varepsilon_j(u)=\frac{1}{\fall{u}{j}}.
\]
Because
\(\varepsilon_j(u)=u^{-j}+O(u^{-j-1})\), the family
\((\varepsilon_j)_{j\geq1}\) is a topological basis of
\(u^{-1}\C[[u^{-1}]]\).  More precisely, every
\(F(u)\in u^{-1}\C[[u^{-1}]]\) has a unique expansion
\[
 F(u)=\sum_{j\geq1}f_j\varepsilon_j(u),
\]
where the sum is understood in the \(u^{-1}\)-adic topology.  We use
the decomposition
\[
 \C((u^{-1}))=\C[u]\oplus u^{-1}\C[[u^{-1}]].
\]
Since multiplication by a polynomial preserves \(\C[u]\), it induces
an operator on the quotient
\(\C((u^{-1}))/\C[u]\), which we identify with
\(u^{-1}\C[[u^{-1}]]\).  Thus, throughout this subsection, polynomial
parts are discarded when a polynomial acts on a reciprocal Newton
series.

To translate the denominator-clearing identity into a recurrence for
the Newton coefficients \(D_j\), we first describe the action of
multiplication by \(u+1\) on the reciprocal falling-factorial basis.

\begin{lemma}
\label{lem:T-operator}
For every \(j\geq1\),
\begin{equation}
 (u+1)\varepsilon_j(u)
 =\varepsilon_{j-1}(u)+j\varepsilon_j(u).
 \label{eq:epsilon-shift}
\end{equation}
Consequently, multiplication by \(u+1\), modulo \(\C[u]\), acts on a
coefficient sequence \(D=(D_j)_{j\geq1}\) by
\begin{equation}
 (\mathcal TD)_j=D_{j+1}+jD_j
 \qquad(j\geq1).
 \label{eq:T-operator}
\end{equation}
For every polynomial \(g\in K[x]\), after extending scalars to \(K\),
\begin{equation}
 \bigl(g(\mathcal T)D\bigr)_j
 =\sum_{l=0}^{\deg g}
   \frac{(\Delta^lg)(j)}{l!}\,D_{j+l},
 \label{eq:T-closed-form}
\end{equation}
where \((\Delta g)(j)=g(j+1)-g(j)\).
\end{lemma}

\begin{proof}
Since \(u+1=(u-j+1)+j\), we have
\[
 \frac{u+1}{\fall{u}{j}}
 =\frac{u-j+1}{u(u-1)\cdots(u-j+1)}
  +\frac{j}{\fall{u}{j}}
 =\varepsilon_{j-1}(u)+j\varepsilon_j(u),
\]
which proves \eqref{eq:epsilon-shift}.  If
\(F(u)=\sum_{j\geq1}D_j\varepsilon_j(u)\), the coefficient of
\(\varepsilon_j\) in \((u+1)F(u)\), modulo polynomials, is therefore
\(D_{j+1}+jD_j\).  This proves \eqref{eq:T-operator}.

For \eqref{eq:T-closed-form}, let \(L_g\) denote the operator on the
right-hand side.  We have \(L_1=\mathrm{id}\) and
\(L_x=\mathcal T\).  It is therefore enough to prove
\(\mathcal T L_g=L_{xg}\).  For \(m\geq0\), the coefficient of
\(D_{j+m}\) in \((\mathcal TL_gD)_j\) is
\[
 \frac{j(\Delta^mg)(j)}{m!}
 +\begin{cases}
 \displaystyle
 \frac{(\Delta^{m-1}g)(j+1)}{(m-1)!},&m\geq1,\\[6pt]
 0,&m=0.
 \end{cases}
\]
The finite-difference Leibniz rule gives
\[
 \Delta^m(xg)(j)
 =j\Delta^mg(j)+m\Delta^{m-1}g(j+1),
\]
where the second term is absent when \(m=0\).  Dividing by \(m!\)
shows that the preceding coefficient is exactly
\((\Delta^m(xg))(j)/m!\).  Hence
\(\mathcal TL_g=L_{xg}\), and induction on \(\deg g\) proves
\eqref{eq:T-closed-form}.
\end{proof}

Clearing the denominator of the Harish-Chandra product now produces
a polynomial annihilator of degree \(N\) for the Newton coefficient
sequence.  Eliminating its coefficients at \(N+1\) consecutive indices
gives the following central determinant identity.

\begin{theorem}
\label{thm:D-form}
For every \(r\geq1\), the following identity holds in \(Z_N\):
\begin{equation}
 \boxed{
 \det\Bigl(\bigl(\mathcal T^kD\bigr)_{r+i}
      \Bigr)_{0\leq i,k\leq N}=0.}
 \label{eq:D-form}
\end{equation}
Explicitly,
\begin{equation}
 \bigl(\mathcal T^kD\bigr)_m
 =\sum_{l=0}^{k}\frac{1}{l!}
  \left(
   \sum_{t=0}^{l}(-1)^{l-t}\binom{l}{t}(m+t)^k
  \right)D_{m+l}.
 \label{eq:T-power-explicit}
\end{equation}
Thus every entry in \eqref{eq:D-form} is an explicit scalar linear
combination of \(D_m,\ldots,D_{m+k}\).
\end{theorem}

\begin{proof}
Set
\[
 Q(u)=\prod_{i=1}^N(u+1-\lambda_i),
 \qquad
 g(x)=\prod_{i=1}^N(x-\lambda_i),
\]
so that \(Q(u)=g(u+1)\).  Equation
\eqref{eq:capelli-series-HC} gives
\begin{equation}
 Q(u)\HC\bigl(\calD_N(u)\bigr)
 =\prod_{i=1}^N(u+1+\lambda_i)\in\C[\lambda_1,\ldots,
 \lambda_N][u].
 \label{eq:polynomial-after-clearing}
\end{equation}
The image of the right-hand side in
\(K((u^{-1}))/K[u]\) is zero.  By
Lemma~\ref{lem:T-operator}, the reciprocal Newton coefficients of the
left-hand side are
\(g(\mathcal T)\HC(D)\).  Hence
\begin{equation}
 \bigl(g(\mathcal T)\HC(D)\bigr)_j=0
 \qquad(j\geq1).
 \label{eq:D-recurrence}
\end{equation}

Write \(e_a(\lambda)\) for the elementary symmetric polynomial of
degree \(a\) in \(\lambda_1,\ldots,\lambda_N\).  Since
\[
 g(x)=\sum_{k=0}^N(-1)^{N-k}e_{N-k}(\lambda)x^k,
\]
equation \eqref{eq:D-recurrence}, evaluated at
\(j=r,r+1,\ldots,r+N\), says that the nonzero vector
\[
 \bigl((-1)^{N-k}e_{N-k}(\lambda)\bigr)_{k=0}^{N}
 \in K^{N+1}
\]
lies in the kernel of
\[
 \left(
  (\mathcal T^k\HC(D))_{r+i}
 \right)_{0\leq i,k\leq N}.
\]
The vector is nonzero because its entry at \(k=N\) is \(1\).
Therefore the determinant of this matrix is zero in \(K\).

The operator \(\mathcal T\) is defined using only integer scalars and
shifts, so
\[
 \HC\bigl((\mathcal T^kD)_m\bigr)
 =(\mathcal T^k\HC(D))_m.
\]
Applying \(\HC\) to the determinant in \eqref{eq:D-form} therefore
gives zero.  The injectivity of \(\HC\) proves the theorem.  Finally,
\eqref{eq:T-power-explicit} is \eqref{eq:T-closed-form} with
\(g(x)=x^k\) and
\[
 (\Delta^lx^k)(m)
 =\sum_{t=0}^{l}(-1)^{l-t}\binom{l}{t}(m+t)^k.
\]
\end{proof}

\begin{example}
\label{ex:D-form-N1}
For \(N=1\), equation \eqref{eq:D-form} becomes
\begin{equation}
 D_rD_{r+2}-D_{r+1}^2+D_rD_{r+1}=0.
 \label{eq:D-form-N1}
\end{equation}
Indeed, \(\HC(D_j)=2\fall{\lambda_1}{j}\) for \(j\geq1\).  Applying
\(\HC\) to the left-hand side of \eqref{eq:D-form-N1} gives
\[
 4\fall{\lambda_1}{r}^{2}(\lambda_1-r)
 \bigl[(\lambda_1-r-1)-(\lambda_1-r)+1\bigr]=0.
\]
Injectivity of \(\HC\) proves the identity in \(Z_1\).
\end{example}

\begin{remark}
Equation \eqref{eq:D-recurrence} is a recurrence of order \(N\) for
\((D_j)_{j\geq1}\), whose coefficients, after expanding the powers of
\(\mathcal T\), are polynomials in \(j\) of degree at most \(N\).  For
\(N=1\), it reduces, after applying \(\HC\), to
\[
 \HC(D_{j+1})=(\lambda_1-j)\HC(D_j).
\]
The coefficients
\((-1)^ae_a(\lambda)\) are symmetric but need not belong to
\(\Gamma_N\); for this reason, \eqref{eq:D-form} eliminates them and
states the result entirely in terms of central elements.  Notice also
that the ordinary rational-series argument cannot be applied directly
to \((D_j)\).  Already for \(N=1\), the ordinary generating series
\[
 1+2\sum_{j\geq1}\fall{\lambda_1}{j}z^j
\]
is not rational over \(\C(\lambda_1)(z)\).  Indeed, rationality would
give a nontrivial constant-coefficient recurrence
\[
 \sum_{k=0}^dq_k\fall{\lambda_1}{n+k}=0
 \qquad(n\gg0),
 \qquad q_k\in\C(\lambda_1).
\]
After division by \(\fall{\lambda_1}{n}\), this would become
\[
 \sum_{k=0}^dq_k
 \prod_{t=0}^{k-1}(\lambda_1-n-t)=0
 \qquad(n\gg0).
\]
The left-hand side is a polynomial in \(n\) which vanishes at
infinitely many integers, hence vanishes identically.  Its summands have
distinct degrees \(0,1,\ldots,d\) in \(n\), so all \(q_k\) must be zero,
a contradiction.
\end{remark}

\subsection{The basic Hankel divisor and strong typicality}

For a rational characteristic function on a \(p|q\)-dimensional
superspace, Khudaverdian and Voronov identify the first basic
\(q\times q\) Hankel determinant with the resultant of its numerator
and denominator; see \cite[Prop.~3, Eq.~(5.6)]{KV2005}.  We now compute
the analogous determinant for the queer Perelomov-Popov series.  The
factorization also distinguishes typical from strongly typical weights:
for \(\q_N\), these notions differ precisely because strong typicality
also excludes \(\lambda_i=0\).

Recall that a weight
\(\lambda=(\lambda_1,\ldots,\lambda_N)\) is typical if
\[
 \lambda_i+\lambda_j\neq0
 \qquad(i\neq j),
\]
and strongly typical if
\[
 \lambda_i+\lambda_j\neq0
 \qquad(1\leq i,j\leq N).
\]
Thus strong typicality is equivalent to typicality together with
\(\lambda_i\neq0\) for every \(i\); see
\cite[Section~2]{FriskMazorchuk2009}.
Set
\begin{equation}
 \Delta_N
 =\det\bigl(b_{i+j-1}\bigr)_{1\leq i,j\leq N}
 =\det\bigl(c_{2(i+j-1)-1}\bigr)_{1\leq i,j\leq N}
 \in Z_N.
 \label{eq:Hankel-Delta}
\end{equation}
Thus \(\Delta_N\) is a polynomial in
\(b_1,\ldots,b_{2N-1}\).

For \(1\leq i\leq N\), put
\begin{equation}
 x_i=\lambda_i(\lambda_i-1),
 \qquad
 y_i=\lambda_i(\lambda_i+1),
 \label{eq:xy-definition}
\end{equation}
and define
\begin{equation}
 P_N(z)=\prod_{j=1}^N(1-y_jz),
 \qquad
 Q_N(z)=\prod_{i=1}^N(1-x_iz).
 \label{eq:PQ-definition}
\end{equation}
With this notation, the queer Perelomov--Popov identity
\eqref{eq:hankel-PP-product} becomes
\begin{equation}
 1-\sum_{r\geq1}\pi_rz^r
 =\frac{P_N(z)}{Q_N(z)}.
 \label{eq:PQ-cyclic-series}
\end{equation}

The partial-fraction expansion of this rational function realizes the
cyclic coefficients as moments of the \(N\) nodes
\(x_1,\ldots,x_N\).

\begin{lemma}
\label{lem:moment}
In the rational-function field \(K\), one has
\begin{equation}
 \pi_r=\sum_{i=1}^Nw_ix_i^{r-1}
 \qquad(r\geq1),
 \qquad
 w_i=
 -\frac{\prod_{j=1}^N(x_i-y_j)}
        {\prod_{j\neq i}(x_i-x_j)}.
 \label{eq:moment-partial-fraction}
\end{equation}
\end{lemma}

\begin{proof}
The elements \(x_1,\ldots,x_N\) are pairwise distinct and nonzero in
the rational-function field \(K\).  Hence the right-hand side of
\eqref{eq:hankel-PP-product} has a partial-fraction decomposition
\[
 \prod_{j=1}^N\frac{1-y_jz}{1-x_jz}
 =C+\sum_{i=1}^N\frac{A_i}{1-x_iz}
\]
for some \(C,A_i\in K\).  Multiplying by \(1-x_iz\) and then setting
\(z=x_i^{-1}\) gives
\begin{align*}
 A_i
 &=\left.
   (1-x_iz)\prod_{j=1}^N\frac{1-y_jz}{1-x_jz}
   \right|_{z=x_i^{-1}}\\
 &=\frac{\prod_{j=1}^N(x_i-y_j)}
         {x_i\prod_{j\neq i}(x_i-x_j)}.
\end{align*}
For \(r\geq1\), the coefficient of \(z^r\) in the partial-fraction
expansion is \(\sum_iA_ix_i^r\), whereas the coefficient of \(z^r\)
in \eqref{eq:hankel-PP-product} is \(-\pi_r\). Thus
\[
 \pi_r=-\sum_iA_ix_i^r
       =\sum_i(-A_ix_i)x_i^{r-1},
\]
which is \eqref{eq:moment-partial-fraction}.
\end{proof}

\begin{theorem}
\label{thm:Delta-resultant-factorization}
The basic Hankel divisor satisfies
\begin{align}
 \HC(\Delta_N)
 &=
 (-1)^{\binom N2}
 \operatorname{Res}(Q_N,P_N),
 \label{eq:Delta-is-resultant}\\
 &=
 (-1)^{\binom N2}
 \prod_{i,j=1}^N(\lambda_i+\lambda_j)
 \prod_{i\neq j}(\lambda_i-\lambda_j-1).
 \label{eq:Delta-product}
\end{align}
In particular, \(\Delta_N\neq0\).
\end{theorem}

\begin{proof}
Let
\[
 V=(x_j^{i-1})_{1\leq i,j\leq N},
 \qquad
 W=\operatorname{diag}(w_1,\ldots,w_N).
\]
By Lemma~\ref{lem:moment}, \((\pi_{i+j-1})_{1\leq i,j\leq N}=VWV^{\mathsf t}\),
and \(\HC(\Delta_N)=\det(\pi_{i+j-1})_{1\leq i,j\leq N}\) because \(\HC\) is an
algebra homomorphism.  Hence
\[
 \HC(\Delta_N)=\det(V)^2\prod_{i=1}^N w_i .
\]
Using \eqref{eq:moment-partial-fraction}, we obtain
\[
 \prod_{i=1}^Nw_i
 =
 (-1)^N
 \frac{\prod_{i,j=1}^N(x_i-y_j)}
      {\prod_i\prod_{j\neq i}(x_i-x_j)}.
\]
Now $\det(V)^2=\prod_{i<j}(x_j-x_i)^2$
and
\[
 \prod_{i=1}^N\prod_{j\neq i}(x_i-x_j)
 =(-1)^{\binom N2}\prod_{i<j}(x_j-x_i)^2.
\]

Substituting the formula for \(w_i\), we obtain
\begin{equation}
 \HC(\Delta_N)
 =(-1)^{N+\binom N2}\prod_{i,j=1}^N(x_i-y_j).
 \label{eq:Delta-xy-product}
\end{equation}

We next compute the resultant.  We use the convention that if
\(f(z)\) has degree \(m\), leading coefficient \(a_f\), and roots
\(\alpha_1,\ldots,\alpha_m\), counted with multiplicity, then
\[
 \operatorname{Res}(f,g)
 =
 a_f^{\deg g}\prod_{k=1}^m g(\alpha_k).
\] The polynomial \(Q_N(z)\) has leading coefficient
\((-1)^N\prod_i x_i\) and roots \(x_1^{-1},\ldots,x_N^{-1}\). Hence
\begin{align*}
 \operatorname{Res}(Q_N,P_N)
 &=
 \left((-1)^N\prod_i x_i\right)^{\!N}
 \prod_{i=1}^N P_N(x_i^{-1})\\
 &=
 \left((-1)^N\prod_i x_i\right)^{\!N}
 \prod_{i=1}^N
 \left(x_i^{-N}\prod_{j=1}^N(x_i-y_j)\right)\\
 &=
 (-1)^N\prod_{i,j=1}^N(x_i-y_j).
\end{align*}
Comparison with \eqref{eq:Delta-xy-product} gives $\HC(\Delta_N)
 =(-1)^{\binom N2}\operatorname{Res}(Q_N,P_N).$

Finally,
\[
 x_i-y_j
 =\lambda_i(\lambda_i-1)-\lambda_j(\lambda_j+1)
 =(\lambda_i+\lambda_j)(\lambda_i-\lambda_j-1).
\]
The factors with \(i=j\) contribute \((-1)^N\), and therefore
\[
 \prod_{i,j=1}^N(x_i-y_j)
 =
 (-1)^N
 \prod_{i,j=1}^N(\lambda_i+\lambda_j)
 \prod_{i\neq j}(\lambda_i-\lambda_j-1).
\]
Substitution into \eqref{eq:Delta-xy-product} proves the asserted
factorization.

The resulting product is a nonzero polynomial in
\(\lambda_1,\ldots,\lambda_N\).  Thus
\(\HC(\Delta_N)\neq0\), and the injectivity of the Harish-Chandra
homomorphism implies \(\Delta_N\neq0\).
\end{proof}

\begin{example}
For \(N=1\), one has \(\Delta_1=c_1\), and
\eqref{eq:Delta-product} gives
\[
 \HC(\Delta_1)=2\lambda_1,
\]
in agreement with \(\HC(c_1)=2p_1=2\lambda_1\).
\end{example}

\begin{remark}
Equation \eqref{eq:Delta-is-resultant} is the queer specialization of
the resultant mechanism in
\cite[Proposition~3, Equation~(5.6)]{KV2005}.  What is specific to
\(\q_N\) is the substitution
\[
x_i=\lambda_i(\lambda_i-1),
\qquad
y_i=\lambda_i(\lambda_i+1),
\]
which factors the resultant into the failure-of-strong-typicality
factor and the additional shifted-resonance factors appearing in
\eqref{eq:Delta-product}.
\end{remark} 

The factorization in \eqref{eq:Delta-product} immediately determines
the vanishing locus of the basic Hankel divisor and its relation to
strong typicality.

\begin{corollary}
\label{cor:Delta-typicality}
Let \(\lambda=(\lambda_1,\ldots,\lambda_N)\in\C^N\).  Then
\(\HC(\Delta_N)(\lambda)=0\) if and only if at least one of the
following conditions holds:
\begin{enumerate}
 \item \(\lambda_i+\lambda_j=0\) for some \(i,j\);
 \item \(\lambda_i-\lambda_j=1\) for some \(i\neq j\).
\end{enumerate}
The first condition is precisely the failure of \emph{strong
typicality}.  By contrast, ordinary typicality only requires
\(\lambda_i+\lambda_j\neq0\) for \(i\neq j\); thus a typical weight with
\(\lambda_i=0\) is still contained in the vanishing locus of
\(\HC(\Delta_N)\).
\end{corollary}

\begin{proof}
The equivalence follows immediately from \eqref{eq:Delta-product}.
For the first factor, separating diagonal and off-diagonal terms gives
\[
 \prod_{i,j}(\lambda_i+\lambda_j)
 =2^N\prod_i\lambda_i
  \prod_{i<j}(\lambda_i+\lambda_j)^2.
\]
It vanishes precisely when \(\lambda_i+\lambda_j=0\) for some pair,
where \(i=j\) is allowed.  The distinction between typical and strongly
typical weights is exactly the distinction between requiring this
nonvanishing for \(i\neq j\) and requiring it for all \(i,j\); see
\cite[Sec.~2]{FriskMazorchuk2009}.
\end{proof}

\subsection{Generic finite generation and reconstruction}

Khudaverdian and Voronov show that the first \(p+q\) coefficients of a
generic \(p|q\)-rational sequence determine its characteristic
function; see \cite[Eq.~(4.1)]{KV2005}.  Their Proposition~3 and the
proof of Theorem~3 identify the relevant Hankel determinant with a
resultant and recover the numerator and denominator by Cramer's rule.
We now apply the same reconstruction mechanism inside the localized
center of \(\U(\q_N)\).  The nonvanishing of \(\Delta_N\) makes the
linear system determining the recurrence coefficients invertible.

\begin{theorem}
\label{thm:generic-finite-generation}
With \(\Delta_N\) as in \eqref{eq:Hankel-Delta},
\begin{equation}
 \boxed{
 Z_N[\Delta_N^{-1}]
 =\C[b_1,\ldots,b_{2N},\Delta_N^{-1}]
 =\C[c_1,c_3,\ldots,c_{4N-1},\Delta_N^{-1}].}
 \label{eq:generic-finite-generation}
\end{equation}
More precisely, there are unique elements
\[
 \rho_1,\ldots,\rho_N
 \in\C[b_1,\ldots,b_{2N},\Delta_N^{-1}]
\]
such that
\begin{equation}
 b_m+\rho_1b_{m-1}+\cdots+\rho_Nb_{m-N}=0
 \qquad(m>N).
 \label{eq:central-reconstruction-recurrence}
\end{equation}
In particular, every \(b_m\) with \(m>2N\) is recovered recursively
from \(b_1,\ldots,b_{2N}\) after \(\Delta_N\) is inverted.
\end{theorem}

\begin{proof}
Set
\[
 \mathcal A_N
 =\C[b_1,\ldots,b_{2N},\Delta_N^{-1}].
\]
Consider the system
\begin{equation}
 b_{N+i}+\rho_1b_{N+i-1}+\cdots+\rho_Nb_i=0,
 \qquad 1\leq i\leq N.
 \label{eq:rho-linear-system}
\end{equation}
Its coefficient matrix is
\[
 \mathsf M=(b_{N+i-j})_{1\leq i,j\leq N}.
\]
Reversing the order of the columns of \(\mathsf M\) gives
\((b_{i+j-1})_{1\leq i,j\leq N}\).  Hence
\begin{equation}
 \det\mathsf M
 =(-1)^{N(N-1)/2}\Delta_N,
 \label{eq:M-determinant}
\end{equation}
which is a unit in \(\mathcal A_N\).  Cramer's rule therefore gives a
unique solution
\(\rho_1,\ldots,\rho_N\in\mathcal A_N\).

Write
\begin{equation}
 1+\eta_1z+\cdots+\eta_Nz^N
 =\prod_{i=1}^N
  \bigl(1-\lambda_i(\lambda_i-1)z\bigr)
 \in K[z]
 \label{eq:eta-denominator}
\end{equation}
for the denominator of \eqref{eq:hankel-PP-product}.  Applying
Lemma~\ref{lem:rational-hankel} to
\(h_0=1\) and \(h_m=-\pi_m\) yields
\begin{equation}
 \pi_m+\eta_1\pi_{m-1}+\cdots+\eta_N\pi_{m-N}=0
 \qquad(m>N).
 \label{eq:p-eta-recurrence}
\end{equation}
For \(m=N+1,\ldots,2N\), these equations show that
\((\eta_1,\ldots,\eta_N)\) solves the Harish--Chandra image of
\eqref{eq:rho-linear-system}.  By
Theorem~\ref{thm:Delta-resultant-factorization}, the determinant of that system is
\[
 (-1)^{N(N-1)/2}\HC(\Delta_N)\neq0
\]
in \(K\).  Uniqueness of the solution in \(K^N\) therefore gives
\begin{equation}
 \HC(\rho_j)=\eta_j
 \qquad(1\leq j\leq N).
 \label{eq:HC-rho-eta}
\end{equation}
Thus, although \(\eta_j\) need not lie in \(\Gamma_N\), it belongs to
the localization
\(\Gamma_N[\HC(\Delta_N)^{-1}]\).  For example,
\[
 \eta_1=-\sum_i\lambda_i(\lambda_i-1)
\]
does not satisfy the cancellation condition defining \(\Gamma_N\).

For every \(m>N\), equations \eqref{eq:p-eta-recurrence} and
\eqref{eq:HC-rho-eta} imply
\[
 \HC\bigl(
  b_m+\rho_1b_{m-1}+\cdots+\rho_Nb_{m-N}
 \bigr)=0.
\]
The extension of \(\HC\) to \(\operatorname{Frac}(Z_N)\) is injective,
so \eqref{eq:central-reconstruction-recurrence} follows.  Since each
\(\rho_j\) lies in \(\mathcal A_N\), induction on \(m\) shows that
\(b_m\in\mathcal A_N\) for every \(m>2N\).  Equation
\eqref{eq:center-generated-by-b} now implies
\[
 Z_N[\Delta_N^{-1}]\subseteq\mathcal A_N.
\]
The reverse inclusion is immediate.  Finally,
\(b_r=c_{2r-1}\) and \(b_{2N}=c_{4N-1}\), which proves the second
equality in \eqref{eq:generic-finite-generation}.

The uniqueness assertion follows already from the first \(N\)
relations \(m=N+1,\ldots,2N\), since their coefficient matrix is
\(\mathsf M\) and \(\det\mathsf M\) is invertible.
\end{proof}

The localized generation theorem has the following representation-
theoretic consequence: away from the divisor defined by \(\Delta_N\),
a central character is determined by finitely many cyclic values.

\begin{corollary}
\label{cor:generic-central-characters}
Let \(\chi,\chi'\colon Z_N\to\C\) be central characters, that is,
unital \(\C\)-algebra homomorphisms.  Suppose that
\(\chi(\Delta_N)\neq0\).  If
\[
 \chi(c_{2r-1})=\chi'(c_{2r-1})
 \qquad(1\leq r\leq2N),
\]
then \(\chi=\chi'\).

In particular, for \(\lambda\in\C^N\), consider the
Harish--Chandra evaluation character
\[
 \chi_\lambda\colon Z_N\longrightarrow\C,
 \qquad
 \chi_\lambda(z)=\HC(z)(\lambda).
\]
Then \(\chi_\lambda(\Delta_N)\neq0\) if and only if \(\lambda\) is
strongly typical and
\[
 \lambda_i-\lambda_j\neq1
 \qquad(i\neq j).
\]
\end{corollary}

\begin{proof}
The two characters agree on \(b_1,\ldots,b_{2N}\), and hence on
\(\Delta_N\), which is a polynomial in
\(b_1,\ldots,b_{2N-1}\).  Therefore
\[
 \chi'(\Delta_N)=\chi(\Delta_N)\neq0,
\]
so both homomorphisms extend uniquely to
\(Z_N[\Delta_N^{-1}]\).  By
Theorem~\ref{thm:generic-finite-generation}, this localization is
generated by \(b_1,\ldots,b_{2N}\) and \(\Delta_N^{-1}\), on which the
two extensions agree.  Restricting them to \(Z_N\) proves
\(\chi=\chi'\).  The last assertion follows from
Corollary~\ref{cor:Delta-typicality}.
\end{proof}

\section*{Acknowledgments}
The authors thank Alexander Molev for his valuable comments and
suggestions.


\begin{thebibliography}{99}

\bibitem{ASS2018}
A. Alldridge, S. Sahi, and H. Salmasian,
\emph{Schur \(Q\)-functions and the Capelli eigenvalue problem for the
Lie superalgebra \(\mathfrak q(n)\)},
in \emph{Representation Theory and Harmonic Analysis on Symmetric
Spaces}, Contemp. Math. \textbf{714}, Amer. Math. Soc., Providence,
RI, 2018, 1-21.
\href{https://doi.org/10.1090/conm/714/14376}
{doi:10.1090/conm/714/14376}.

\bibitem{ChengWang2012}
S.-J. Cheng and W. Wang,
\emph{Dualities and Representations of Lie Superalgebras},
Graduate Studies in Mathematics \textbf{144}, Amer. Math. Soc.,
Providence, RI, 2012.

\bibitem{DP2025}
A. Das and S. Pattanayak,
\emph{Generalized Casimir operators for loop Lie superalgebras},
Canad. Math. Bull. \textbf{68} (2025), no.~3, 744--760.
\href{https://doi.org/10.4153/S0008439525000050}
{doi:10.4153/S0008439525000050}.

\bibitem{EKK2025}
S. Erat, A. S. Kannan, and S. Kanungo,
\emph{Mixed tensor products, Capelli Berezinians, and Newton's formula
for \(\mathfrak{gl}(m|n)\)},
Transform. Groups (2025), published online.
\href{https://doi.org/10.1007/s00031-025-09908-0}
{doi:10.1007/s00031-025-09908-0}.

\bibitem{FriskMazorchuk2009}
A. Frisk and V. Mazorchuk,
\emph{Regular strongly typical blocks of \(\mathcal O^{\mathfrak q}\)},
Comm. Math. Phys. \textbf{291} (2009), 533-542.
\href{https://doi.org/10.1007/s00220-009-0799-z}
{doi:10.1007/s00220-009-0799-z}.

\bibitem{GN2019}
T. A. Grigoryev and M. L. Nazarov,
\emph{An analogue of the Perelomov--Popov formula for the Lie
superalgebra \(Q(N)\)},
Funct. Anal. Appl. \textbf{53} (2019), 304-308.
\href{https://doi.org/10.1134/S0016266319040075}
{doi:10.1134/S0016266319040075}.

\bibitem{Ivanov2005}
V. N. Ivanov,
\emph{Interpolation analogues of Schur \(Q\)-functions},
J. Math. Sci. (N.Y.) \textbf{131} (2005), 5495--5507.
\href{https://doi.org/10.1007/s10958-005-0422-6}
{doi:10.1007/s10958-005-0422-6}.

\bibitem{KT1997}
I. Kantor and I. Trishin,
\emph{The algebra of polynomial invariants of the adjoint
representation of the Lie superalgebra \(\gl(m|n)\)},
Comm. Algebra \textbf{25} (1997), no.~7, 2039-2070.
\href{https://doi.org/10.1080/00927879708825971}
{doi:10.1080/00927879708825971}.

\bibitem{KT1999}
I. Kantor and I. Trishin,
\emph{On the Cayley--Hamilton equation in the supercase},
Comm. Algebra \textbf{27} (1999), no.~1, 233-259.
\href{https://doi.org/10.1080/00927879908826430}
{doi:10.1080/00927879908826430}.

\bibitem{KM2025}
I. Kashuba and A. Molev,
\emph{Universal Capelli identities and quantum immanants for the queer
Lie superalgebra}, arXiv:2512.21631 (2025).
\url{https://arxiv.org/abs/2512.21631}.

\bibitem{KV2005}
H. M. Khudaverdian and Th. Th. Voronov,
\emph{Berezinians, exterior powers and recurrent sequences},
Lett. Math. Phys. \textbf{74} (2005), no.~2, 201-228.
\href{https://doi.org/10.1007/s11005-005-0025-7}
{doi:10.1007/s11005-005-0025-7}.

\bibitem{LuoWang2025}
Y. Luo and Y. Wang,
\emph{The Schur--Weyl duality and invariants for classical Lie
superalgebras},
Canad. J. Math.,(2025), 1-36.
\href{https://doi.org/10.4153/S0008414X25101855}
{doi:10.4153/S0008414X25101855}.

\bibitem{Macdonald1995}
I. G. Macdonald,
\emph{Symmetric Functions and Hall Polynomials}, 2nd ed.,
Oxford University Press, Oxford, 1995.

\bibitem{Nazarov1991}
M.~L.~Nazarov,
\emph{Quantum Berezinian and the classical Capelli identity},
Lett. Math. Phys. \textbf{21} (1991), 123--131.
\href{https://doi.org/10.1007/BF00401646}
{doi:10.1007/BF00401646}.

\bibitem{Nazarov1997}
M. Nazarov,
\emph{Capelli identities for Lie superalgebras},
Ann. Sci. \'{E}cole Norm. Sup. (4) \textbf{30} (1997), no.~6,
847-872.
\href{https://doi.org/10.1016/S0012-9593(97)89941-7}
{doi:10.1016/S0012-9593(97)89941-7}.

\bibitem{NazarovSergeev2006}
M. Nazarov and A. Sergeev,
\emph{Centralizer construction of the Yangian of the queer Lie
superalgebra},
in \emph{Studies in Lie Theory}, Progr. Math. \textbf{243},
Birkh\"auser Boston, Boston, MA, 2006, 417-441.
\href{https://doi.org/10.1007/0-8176-4478-4_17}
{doi:10.1007/0-8176-4478-4\_17}.

\bibitem{Sergeev1983}
A. N. Sergeev,
\emph{The centre of enveloping algebra for Lie superalgebra
\(Q(n,\C)\)},
Lett. Math. Phys. \textbf{7} (1983), 177-179.
\href{https://doi.org/10.1007/BF00400431}
{doi:10.1007/BF00400431}.

\bibitem{Sergeev1999}
A. Sergeev,
\emph{The invariant polynomials on simple Lie superalgebras},
Represent. Theory \textbf{3} (1999), 250-280.
\href{https://doi.org/10.1090/S1088-4165-99-00077-1}
{doi:10.1090/S1088-4165-99-00077-1}.


\end{thebibliography}
\end{document}